\documentclass[10pt,a4paper,reqno]{amsart}

\usepackage{amsmath,amsthm,amssymb}
{}
\usepackage{tikz}

\usepackage{graphicx,epstopdf}

\usepackage{color}

\usepackage{amsfonts}

\usepackage[T1]{fontenc}

\usepackage{enumerate}

\usepackage{mathtools}

\usepackage[utf8]{inputenc}

\usepackage{graphicx}
\usepackage{amssymb}
\usepackage{esint}

\makeatletter
\newdimen\rh@wd
\newdimen\rh@hta
\newdimen\rh@htb
\newbox\rh@box
\def\rh@measure#1{\setbox\rh@box=\hbox{$#1$}\rh@wd=\wd\rh@box \rh@hta=\ht\rh@box}

\def\widecheck#1{\rh@measure{#1}%
  \setbox\rh@box=\hbox{$\widehat{\vrule height \rh@hta width\z@ \kern\rh@wd}$}%
  \rh@htb=\ht\rh@box \advance\rh@htb\rh@hta \advance\rh@htb\p@
  \ooalign{$\vrule height \ht\rh@box width\z@ #1$\cr
           \raise\rh@htb\hbox{\scalebox{1}[-1]{\box\rh@box}}\cr}}
\makeatother

\usepackage{palatino}

\usepackage[abbrev,backrefs]{amsrefs}

\usepackage{color}
\usepackage{todonotes}

\usepackage[nameinlink]{cleveref}
\usepackage{constants}

\newcommand{\abs}[1]{\mathopen\lvert#1\mathclose\rvert}

\newcommand{\Bigabs}[1]{\Bigl\lvert#1\Bigr\rvert}
\newcommand{\biggabs}[1]{\biggl\lvert#1\biggr\rvert}
\newcommand{\norm}[1]{\mathopen\lVert#1\mathclose\rVert}

\newcommand{\N}{{\mathbb N}}
\newcommand{\R}{{\mathbb R}}

\newcommand{\CC}{{\mathbb C}}

\newcommand{\cH}{\mathcal{H}}

\newcommand{\e}{{\mathrm{e}}}

\DeclareMathOperator{\supp}{supp}

\DeclareMathOperator{\Dom}{dom}

\newcommand{\dif}{\,\mathrm{d}}

\newcommand{\Lebesgue}[1]{\widehat{#1}}

\newcommand{\loc}{_\mathrm{loc}}

\theoremstyle{plain}
\newtheorem{proposition}{Proposition}[section]
\newtheorem{lemma}[proposition]{Lemma}
\newtheorem{theorem}[proposition]{Theorem}
\newtheorem{corollary}[proposition]{Corollary}

\theoremstyle{definition}

\theoremstyle{remark}
\newtheorem{example}{Example}[section]
\newtheorem{openproblem}[proposition]{Open problem}

\newtheorem*{Claim}{Claim}

\numberwithin{equation}{section}

\title[The precise representative for the gradient of the Riesz~potential]{The precise representative for the gradient of the Riesz potential of a finite measure}

\author{Julià Cuf\'i}
\address{
Julià Cufí\hfill\break\indent
Universitat Autònoma de Barcelona\hfill\break\indent
Departament de Matemàtiques\hfill\break\indent
08193 Bellaterra, Barcelona, Catalonia}
\email{jcufi@mat.uab.cat}

\author{Augusto C. Ponce}
\address{
Augusto C. Ponce\hfill\break\indent
Université catholique de Louvain\hfill\break\indent
Institut de Recherche en Mathématique et Physique\hfill\break\indent
Chemin du cyclotron 2, L7.01.02\hfill\break\indent
1348 Louvain-la-Neuve, Belgium}
\email{Augusto.Ponce@uclouvain.be}

\author{Joan Verdera}
\address{
Joan Verdera\hfill\break\indent
Universitat Autònoma de Barcelona\hfill\break\indent
Departament de Matemàtiques\hfill\break\indent 
and\hfill\break\indent
Centre de Recerca Matemàtica\hfill\break\indent
08193 Bellaterra, Barcelona, Catalonia}
\email{jvm@mat.uab.cat}

\begin{document}

\begin{abstract}
Given a finite nonnegative Borel measure \(m\) in \(\mathbb{R}^{d}\), we identify the Lebesgue set \(\mathcal{L}(V_{s}) \subset \mathbb{R}^{d}\) of the vector-valued function
\[
V_{s}(x)
= \int_{\mathbb{R}^{d}}\frac{x - y}{|x - y|^{s + 1}} \mathrm{d}m(y),
\]
for any order \(0 < s < d\).{}
We prove that \(a \in \mathcal{L}(V_{s})\) if and only if the integral above has a principal value at \(a\) and
\[
\lim_{r \to 0}{\frac{m(B_{r}(a))}{r^{s}}} = 0.
\]
In that case, the precise representative of \(V_{s}\) at \(a\) coincides with the principal value of the integral. 
We also study the existence of Lebesgue points for the Cauchy integral of the intrinsic probability measure associated with planar Cantor sets, which leads to challenging new questions.
\end{abstract}



\maketitle

\section{Introduction}

Given \(0 < s < d\) and a finite nonnegative Borel measure \(m\) in \(\R^{d}\), we consider in \(\R^{d}\) the vector-valued function
\begin{equation}
\label{eqRiezDerivative}
V_{s}(x) 
\vcentcolon= \int_{\R^{d}} \frac{x - y}{\abs{x - y}^{s + 1}} \dif m(y) 
\end{equation}
that is well-defined for every point \(x\) in the set
\[{}
\Dom{V_{s}}
\vcentcolon=  \bigg\{ x \in \R^{d} : \int_{\R^{d}} \frac{\dif m(y)}{\abs{x - y}^{s}} < \infty \bigg\}.
\] 
The domain \(\Dom{V_{s}}\) is rather large as its complement \(\R^{d} \setminus \Dom{V_{s}}\) is negligible with respect to the Lebesgue measure.
Indeed, as a consequence of Fubini's theorem one has
\[{}
\int_{\R^{d}} \biggl( \int_{\R^{d}} \frac{\dif m(y)}{\abs{x - y}^{s}}  \biggr) \e^{-\abs{x}^{2}} \dif x 
< \infty.
\]
A finer analysis shows that the Hausdorff dimension of  \(\R^{d} \setminus \Dom{V_{s}}\) is actually
not greater that \(s\)\,; see Theorem~1 in Chapter~4 of \cite{Carleson}.

Up to multiplicative constants, \(V_{s}\) is the gradient of the Riesz potential of order \(s - 1\) and, in particular, \(V_{d - 1}\) is the gradient of the Newtonian potential
\[{}
\int_{\R^{d}}  \frac{\dif m(y)}{\abs{x - y}^{d-2}} \quad \text{for \(x \in \R^d\)}
\]
generated by a distribution of mass described by \(m\).{} 
Hence, for every function \(u \in L^{1}\loc(\R^{d})\) such that
\begin{equation}
\label{eq-232}
- \Delta u = m
\quad \text{in the sense of distributions in \(\R^{d}\),}
\end{equation}
by Weyl's lemma on weakly harmonic functions it follows that
\[{}
\nabla u = c V_{d - 1} + \Phi{}
\quad \text{almost everywhere in \(\R^{d}\),}
\]
where \(\Phi : \R^{d} \to \R^{d}\) is the gradient of a harmonic function and \(c \in \R\) is a nonzero constant that depends on \(d\).
As a result, a reasonable pointwise definition of \(\nabla u\) for any solution of \eqref{eq-232} can be obtained from a
good understanding of \(V_{d - 1}\).

More generally, we are interested in the existence of the \emph{precise representative} of \(V_{s}\) at a given point \(a \in \R^{d}\).{}
We recall that the precise representative in this case is a vector \(\alpha \in \R^{d}\) that satisfies
\begin{equation}
\label{eq-226}
\lim_{r \to 0}{\fint_{B_{r}(a)}{\abs{V_{s} - \alpha}}} 
= 0,
\end{equation}
where the symbol \(\fint\) stands for the average integral with respect to the $d$-dimensional Lebesgue measure and $B_r(a)$
for the ball of radius $r$ centered at $a$.
When such an \(\alpha\) exists, we say that \(a\) is a \emph{Lebesgue point} of \(V_{s}\) and we denote 
\[{}
\Lebesgue{V_{s}}(a)
\vcentcolon= \alpha.
\]
It follows from the fundamental property \eqref{eq-226} that
\[{}
\Lebesgue{V_{s}}(a)
= \lim_{r \to 0}{\fint_{B_{r}(a)}{V_{s}}}.
\]
As a consequence of the classical Lebesgue Differentiation Theorem, \(\Lebesgue{V_{s}}(a)\) exists for almost every \(a \in \R^{d}\) with respect to the Lebesgue measure and
\[{}
\Lebesgue{V_{s}} = V_{s}
\quad \text{almost everywhere in \(\R^{d}\).}
\]

When the function one is dealing with has better properties, like being in some Sobolev space, 
the exceptional set \(\mathcal{E}\) (that is, the complement of the set \(\mathcal{L}\) of Lebesgue points) is typically smaller
than merely having zero Lebesgue measure; see e.g.~Section~4.8 in \cite{Evans-Gariepy} and Chapter~8 in \cite{Ponce:2016}.
It is then natural to expect that the same property holds for the exceptional set \(\mathcal{E}(V_{s})\) of  
the potential \(V_{s}\).

This note stems from the recent work~\cite{Cufi-Verdera} by the first and third authors concerning capacitary differentiability of the Newtonian potential; see also \cite{Verdera}.
In analogy with \cite{Cufi-Verdera}, we show that existence of a principal value for the integral, combined with a density property of \(m\) at \(a\), allows one to decide whether \(a\) is a Lebesgue point of \(V_{s}\).

Before stating our result, we recall that the principal value of \(V_{s}\) at \(a\) is
\[{}
\mathrm{p.v.}\,V_{s}(a)
\vcentcolon= 
\lim_{\epsilon \to 0}{\int_{\R^{d} \setminus B_{\epsilon}(a)} \frac{a - y}{\abs{a - y}^{s + 1}} \dif m(y)}
\]
whenever this limit exists.
We relate the question of existence of a principal value for \(V_{s}\) with that of a precise representative in the following

\begin{theorem}
	\label{theoremRieszLebesgue}
	A point \(a \in \R^{d}\) is a Lebesgue point of \(V_{s}\)
	if and only if the principal value of \(V_{s}\) exists at \(a\) and
	\begin{equation}
	\label{eq-328}
	\lim_{r \to 0}{\frac{m(B_{r}(a))}{r^{s}}} = 0.
	\end{equation}
	One then has
	\[{}
	\Lebesgue{V_{s}}(a)
	= \mathrm{p.v.}\,V_{s}(a).
	\]
\end{theorem}

	We prove Theorem~\ref{theoremRieszLebesgue} in Sections~\ref{sectionConverse} and~\ref{sectionDirect}.
	That every \(a \in \Dom{V_{s}}\) is a Lebesgue point of \(V_{s}\) and \(\Lebesgue{V_{s}}(a) = V_{s}(a)\) can be seen independently from Theorem~\ref{theoremRieszLebesgue} by observing that, for every \(r > 0\) and \(y \in \R^{d}\) with \(y \ne a\),	
	\[{}
	 \fint_{B_{r}(a)} \frac{\dif x}{|x - y|^{s}}
	 \le \frac{C_{d, s}}{|a - y|^{s}}
	\]
	and then
	\[{}
	\fint_{B_{r}(a)} \biggabs{\frac{x-y}{|x-y|^{s+1}} - \frac{a - y}{|a - y|^{s+1}}} \dif x
	\le \frac{C_{d, s} + 1}{\abs{a - y}^{s}}
	\quad \text{for \(y \ne a\).}
	\]
	An application of Fubini's Theorem and the Dominated Convergence Theorem now gives
	\[{}
	\fint_{B_{r}(a)}{\abs{V_{s} - V_{s}(a)}}
	\le \int_{\R^{d}} \biggl( \fint_{B_{r}(a)} \biggabs{\frac{x-y}{|x-y|^{s+1}} - \frac{a - y}{|a - y|^{s+1}}} \dif x \biggr) \dif m(y)
	\to 0
	\quad \text{as \(r \to 0\).}
	\]
	As \(\Dom{V_{s}}\) is contained in the set of Lebesgue points of \(V_{s}\), we get
	\begin{equation}
	\label{eqHausdorffDimension}
	\dim_{\mathcal{H}}{(\mathcal{E}(V_{s}))} \le{}
	\dim_{\mathcal{H}}{(\R^{d} \setminus \Dom{V_{s}})} \le s.
	\end{equation}
	
	To obtain a more precise quantification of the size of  \(\mathcal{E}(V_{s})\), we observe that, by Theorem~\ref{theoremRieszLebesgue}, failure of having Lebesgue points may occur either when the principal value of \(V_{s}\) does not exist at \(a\) or
	\[{}
	\limsup_{r \to 0}{\frac{m(B_{r}(a))}{r^{s}}} > 0.
	\]
	The latter condition holds on a set of points \(a\) with at most \(\sigma\)-finite Hausdorff measure \(\mathcal{H}^{s}\).{}

	The question of existence of principal values is more subtle and has been investigated by Mattila and the third author~\cite{Mattila-Verdera}.
	The answer involves a capacity \(\kappa_{s}\) that is defined on every compact subset \(E \subset \R^{d}\) as 
	\[{}
	{\kappa_{s}(E) = \sup{\mu(\R^{d})},}
	\]
	where the supremum is computed over all finite nonnegative Borel measures \(\mu\) supported in \(E\) such that
	\begin{enumerate}[(a)]
		\item \(\mu(A) = 0\) for every Borel set \(A \subset \R^{d}\) with \(\sigma\)-finite \(\cH^{s}\) measure,{}
		\item the maximal Riesz transform \(R_{s}^{*}\) is a bounded operator from \(L^{2}(\mu)\) into itself with
		\[{}
		\norm{R_{s}^{*}(f\mu)}_{L^{2}(\mu)}
		\le \norm{f}_{L^{2}(\mu)}
		\quad \text{for every \(f \in (L^{1} \cap L^{2})(\mu)\)},
		\] 
		where 
	\[{}
	R^{*}_{s}(f\mu)(x) \vcentcolon= \sup_{\epsilon > 0}{\biggl|\int_{\R^{d} \setminus B_{\epsilon}(x)} \frac{x - y}{\abs{x - y}^{s + 1}} f(y) \dif \mu(y) \biggr|}
	\quad \text{for every \(x \in \R^{d}\).}
	\]
	\end{enumerate}
	
	As a consequence of~\cite{Mattila-Verdera} and using a straightforward adaptation of the argument in \cite{Cufi-Verdera}, we show that
	
\begin{theorem}
	\label{theoremCapacity}
	For every compact subset \(E \subset \mathcal{E}(V_{s})\), one has
	\(
	\kappa_{s}(E) = 0.
	\)
\end{theorem}
	
	Other properties of \(\kappa_{s}\) for any \(0 < s < d\) 
	are presented in Section~\ref{sectionCapacity}, where we also prove Theorem~\ref{theoremCapacity} and explain why the latter implies \eqref{eqHausdorffDimension}.
	A variant of $\kappa_s$ was introduced by Prat in \cite{Prat}, where the author obtains analogous estimates for the Hausdorff dimension of sets of zero capacity via comparison between capacity and Hausdorff content.

	A deep result by Ruiz~de~Villa and Tolsa~\cite{RuizdeVilla-Tolsa} ensures that \(\kappa_{d - 1}\) is equivalent to the \(C^{1}\)-harmonic capacity.
	Using this interpretation, the fact that \(\kappa_{d - 1}(E) = 0\) means that every \(C^{1}\)~function in \(\R^{d}\) that is 
	harmonic in \(\R^{d} \setminus E\) must be harmonic in the entire space \(\R^{d}\).
	In dimension \(d = 2\), it follows from work by Tolsa~\cite{Tolsa:2004} that the $C^1$-harmonic capacity is equivalent to the classical continuous analytic capacity.
	We then deduce from Theorem~\ref{theoremCapacity} for \(s = 1\) that
	
\begin{corollary}
	For every finite complex Borel measure \(\nu\) in \(\mathbb{C}\), the exceptional set for Lebesgue points of the Cauchy integral of \(\nu\), defined by convolution as
$$
C(\nu)(z) = \int_{\mathbb{C}} \frac{\dif\nu(w)}{z-w}{}
\quad \text{for almost every \(z \in \mathbb{C}\),}
$$
has zero continuous analytic capacity.
\end{corollary}

The preceding corollary is sharp in the sense that there exists a compact set of zero continuous analytic capacity that is exactly the exceptional set of the Cauchy integral of some finite measure. 
For instance, take as $E$ the corner quarters planar Cantor set, which has positive and finite one dimensional Hausdorff measure $m$.
Then, the Cauchy integral $C(m)$ is continuous in the complement of $E$ and each point of $E$ is an exceptional point of $C(m)$ because \eqref{eq-328} is not satisfied; see Proposition~\ref{propositionGI}.
It would be interesting to know whether the same property holds for every compact set with zero continuous analytic capacity:

\begin{openproblem}
	Given a compact set \(E \subset \CC\) with zero continuous analytic capacity, does there exist a finite measure \(\nu\) such that the set of non-Lebesgue points of \(C(\nu)\) coincides with \(E\)\,?
\end{openproblem}

In Sections~\ref{sectionCantor} and~\ref{sectionCantorProof}, we show that, for Cantor sets \(E \subset \CC\) satisfying \eqref{eq-328} with \(s = 1\) and equipped with their intrinsic probability measure \(\mu\), the Cauchy integral \(C(\mu)\) always has some Lebesgue point in \(E\)\,; see Theorem~\ref{theoremCantorPV}. 
Hence, these sets \(E\) cannot be identified as the exceptional set of \(C(\mu)\) for such measures \(\mu\), but we do not rule out the existence of another measure \(\nu\) with the desired property.

\section{Proof of the reverse implication of Theorem~\ref{theoremRieszLebesgue}}
\label{sectionConverse}

We first need a couple of estimates related to \(V_{s}\).

\begin{lemma}
	\label{lemmaEstimateAverage}
	For every \(a \in \R^{d}\) and every \(r > 0\),{}
	\[{}
	\fint_{B_{r}(a)} \biggl( \int_{B_{2r}(a)} \frac{ \dif m(y)}{\abs{z - y}^{s}} \biggr) \dif z
	\le C \, \frac{m(B_{2r}(a))}{r^{s}},
	\]
	for some constant \(C > 0\) depending on \(d\).
\end{lemma}

\begin{proof}
	By Fubini's theorem,
	\[{}
	\fint_{B_{r}(a)} \biggl( \int_{B_{2r}(a)} \frac{ \dif m(y)}{\abs{z - y}^{s}} \biggr) \dif z
	=
	\int_{B_{2r}(a)} \biggl( \fint_{B_{r}(a)} \frac{ \dif z}{\abs{z - y}^{s}} \biggr) \dif m(y).
	\]
	For every \(y \in \R^{d}\),
	\[{}
	\fint_{B_{r}(a)} \frac{ \dif z}{\abs{z - y}^{s}}
	\le \fint_{B_{r}(a)} \frac{ \dif z}{\abs{z - a}^{s}}
	= \frac{C}{r^{s}},
	\]
	which gives the conclusion.
\end{proof}

\begin{lemma}
	\label{lemmaEstimateMeanValue}
	For every \(a \in \R^{d}\) and every \(r > 0\),{}
	\[{}
	\fint_{B_{r}(a)}{\biggl( \int_{\R^{d} \setminus B_{2r}(a)} \biggabs{\frac{z - y}{\abs{z - y}^{s + 1}} - \frac{a - y}{\abs{a - y}^{s + 1}}} \dif m(y)}\biggr) \dif z
	\le C' r \int_{2r}^{\infty} \frac{m(B_{t}(a))}{t^{s + 2}} \dif t,
	\]	
	for some constant \(C' > 0\) depending on \(d\).
\end{lemma}

\begin{proof}
	To simplify the notation, we may assume that \(a = 0\).
	For \(y \in \R^{d} \setminus B_{2r}(0)\) and \(z \in B_{r}(0)\), by the Mean Value Theorem there exists 
	\(0 \le \theta \le 1\) such that
	\[{}
	\biggabs{\frac{z - y}{\abs{z - y}^{s + 1}} + \frac{y}{\abs{y}^{s + 1}}}
	\le \frac{\C \abs{z}}{\abs{\theta z - y}^{s + 1}} 
	\le \frac{\Cl{cte-286} r}{\abs{y}^{s + 1}}. 
	\]
	We thus have
	\begin{multline*}
	\fint_{B_{r}(0)} \biggl( \int_{\R^{d} \setminus B_{2r}(0)} \biggabs{\frac{z - y}{\abs{z - y}^{s + 1}} + 
	\frac{y}{\abs{y}^{s + 1}}} \dif m(y) \biggr) \dif z\\
	\le \Cr{cte-286} r \fint_{B_{r}(0)} \biggl( \int_{\R^{d} \setminus B_{2r}(0)} \frac{ \dif m(y)}{\abs{y}^{s + 1}} \biggr) \dif z = \Cr{cte-286} r \int_{\R^{d} \setminus B_{2r}(0)} \frac{ \dif m(y)}{\abs{y}^{s + 1}}.
	\end{multline*}
	Using Cavalieri's principle, see e.g.~Corollary~2.2.34 in \cite{Willem},  one gets
	\[{}
	\int_{\R^{d} \setminus B_{2r}(0)} \frac{\dif m(y)}{\abs{y}^{s + 1}}
	= (s + 1) \int_{2r}^{\infty} \frac{m(B_{t}(0))}{t^{s + 2}} \dif t
	\]
	and the lemma follows.	
\end{proof}

To handle the limit as \(r \to 0\) of the right-hand side in the estimate in Lemma~\ref{lemmaEstimateMeanValue},
we need an elementary fact from Real Analysis:

\begin{lemma}
	\label{lemmaLimit}
	Let \(h : (0, \infty){} \to \R\) be a bounded measurable function.
	If\/ \(\lim\limits_{t \to 0}{h(t)} = 0\), then, for every \(\beta > 0\),{}
	\[{}
	\lim_{r \to 0}{r^{\beta} \int_{r}^{\infty} \frac{h(t)}{t^{1 + \beta}} \dif t} 
	= 0.
	\]
\end{lemma}

\resetconstant
\begin{proof}
	Given \(\epsilon > 0\), take \(\delta > 0\) such that
	\(\abs{h(t)} \le \epsilon\) for every \(0 < t \le \delta\).
	For \(0 < r < \delta\), we then have
	\[{}
	\biggl|\int_{r}^{\infty} \frac{h(t)}{t^{1 + \beta}} \dif t\biggr|
	\le \biggl|\int_{r}^{\delta}\biggr| + \biggl|\int_{\delta}^{\infty}\biggr|
	\le \frac{\epsilon}{\beta} \Bigl(\frac{1}{r^{\beta}} + \frac{\norm{h}_{L^{\infty}}}{\delta^{\beta}} \Bigr).
	\]
	Therefore,
	\[{}
	\limsup_{r \to 0}{r^{\beta} \biggl|\int_{r}^{\infty} \frac{h(t)}{t^{1 + \beta}} \dif t \biggr| }
	\le \frac{\epsilon}{\beta}.
	\]
	Since this inequality holds for every \(\epsilon > 0\), the limit equals zero.
\end{proof}

\resetconstant
\begin{proof}[Proof of Theorem~\ref{theoremRieszLebesgue}] ``\(\Longleftarrow\)''.
	We assume that the principal value \(\mathrm{p.v.}\, V_{s}(a)\) exists and the limit \eqref{eq-328} holds.
	For almost every \(z \in \R^{d}\),{}
	\begin{multline*}
	\biggabs{V_{s}(z) - \int_{\R^{d} \setminus B_{2r}(a)} \frac{a - y}{\abs{a - y}^{s + 1}} \dif m(y)}\\
	\le \biggabs{\int_{\R^{d} \setminus B_{2r}(a)} \frac{z - y}{\abs{z - y}^{s + 1}} \dif m(y) - \int_{\R^{d} \setminus B_{2r}(a)} \frac{a - y}{\abs{a - y}^{s + 1}} \dif m(y) }\\
	+ \biggabs{\int_{B_{2r}(a)} \frac{z - y}{\abs{z - y}^{s + 1}} \dif m(y)}.
	\end{multline*}
	Hence,
	\begin{multline*}
	\fint_{B_{r}(a)}\biggabs{V_{s}(z) - \int_{\R^{d} \setminus B_{2r}(a)} \frac{a - y}{\abs{a - y}^{s + 1}} \dif m(y)} \dif z\\
	\le \fint_{B_{r}(a)} \biggl(\int_{\R^{d} \setminus B_{2r}(a)} \biggabs{\frac{z - y}{\abs{z - y}^{s + 1}} \dif m(y) - \frac{a - y}{\abs{a - y}^{s + 1}}} \dif m(y)  \biggr)\dif z\\
	+ \fint_{B_{r}(a)} \biggl( \int_{B_{2r}(a)} \frac{ \dif m(y)}{\abs{z - y}^{s}} \biggr) \dif z.
	\end{multline*}
	We then have by Lemmas~\ref{lemmaEstimateAverage} and \ref{lemmaEstimateMeanValue},
	\begin{multline}
	\label{eq-334}
	\fint_{B_{r}(a)}\biggabs{V_{s}(z) - \int_{\R^{d} \setminus B_{2r}(a)} \frac{a - y}{\abs{a - y}^{s + 1}} \dif m(y)} \dif z
	\le \Cl{cte-335} \biggl( r \int_{2r}^{\infty} \frac{m(B_{t}(a))}{t^{s + 2}} \dif t + \frac{m(B_{2r}(a))}{r^{s}} \biggr).
	\end{multline}
	Denoting by \(\alpha\) the principal value of \(V_{s}\) at \(a\), we get
	\begin{multline*}
	\fint_{B_{r}(a)}\abs{V_{s}(z) - \alpha} \dif z\\
	\le \Cr{cte-335} \biggl( r \int_{2r}^{\infty} \frac{m(B_{t}(a))}{t^{s + 2}} \dif t + \frac{m(B_{2r}(a))}{r^{s}} \biggr)
	 + \biggabs{\int_{\R^{d} \setminus B_{2r}(a)} \frac{a - y}{\abs{a - y}^{s + 1}} \dif m(y) - \alpha}.
	\end{multline*}
	As \(r \to 0\), the last term converges to zero by definition of principal value.
	The quantity in parentheses also converges to zero from the assumption on the measure \(m\), as we can apply Lemma~\ref{lemmaLimit} with \(\beta = 1\) and \(h(t) = m(B_{t}(a))/t^{s}\) for \(t > 0\).
\end{proof}

\section{Proof of the direct implication of Theorem~\ref{theoremRieszLebesgue}}
\label{sectionDirect}

We begin with the following estimate: 

\begin{proposition}
	\label{lemmaEstimateDensity}
	Let \(0 < s \le d - 1\).{}
	For every \(\alpha \in \R^{d}\) and every \(r > 0\),
	\[{}
	\frac{m(B_{r}(a))}{r^{s}}
	\le C \fint_{B_{2r}(a)}{\abs{V_{s} - \alpha}},
	\]
	for some constant \(C > 0\) depending on \(d\).
\end{proposition}
\begin{proof}
	We may assume that \(a = 0\). 
	We take the inner product of \(V_{s}(x)\) with \(x/\abs{x}\) and integrate over the ball \(B_{2r}(0)\) for any \(r > 0\).{}
	By integrability of the function \((x, y) \mapsto {1}/{\abs{x - y}^{s}}\) with respect to the product measure \(\cH^{d} \otimes m\) in \(B_{2r}(0) \times \R^{d}\) we can interchange the order of integration:	
	\begin{equation}
		\label{eq-487}
	\int_{B_{2r}(0)}{V_{s}(x) \cdot \frac{x}{\abs{x}} \dif x}
	= \int_{\R^{d}} \biggl( \int_{B_{2r}(0)} \frac{x - y}{\abs{x - y}^{s + 1}} \cdot \frac{x}{\abs{x}} \dif x \biggr) \dif m(y). 
	\end{equation}
	By the integration formula in polar coordinates, 
	\begin{equation}
		\label{eq-494}
	\int_{B_{2r}(0)} \frac{x - y}{\abs{x - y}^{s + 1}} \cdot \frac{x}{\abs{x}} \dif x
	= \int_{0}^{2r}\biggl( \int_{\partial B_{\rho}(0)} \frac{x - y}{\abs{x - y}^{s + 1}} \cdot \frac{x}{\abs{x}} \dif\sigma(x) \biggr) \dif \rho,
	\end{equation}
	where \(\sigma\) denotes the surface measure on the sphere \(\partial B_{\rho}(0)\).{}

	We claim that
	\begin{equation}
		\label{eq-588}
	\int_{\partial B_{\rho}(0)} \frac{x - y}{\abs{x - y}^{s + 1}} \cdot \frac{x}{\abs{x}} \dif\sigma(x){}
	\ge \epsilon \, \rho^{d - s - 1} \chi_{B_{\rho}(0)}(y),
	\end{equation}
	for some \(\epsilon > 0\) independent of \(\rho\) and \(y\).{}
	To this end, note that
	\[{}
	\frac{x - y}{\abs{x - y}^{s + 1}} = \nabla g(x)
	\quad \text{where \(g(x) \vcentcolon= - \frac{1}{(s - 1)} \frac{1}{\abs{x - y}^{s - 1}}\)}.
	\]
	When \(s < d - 1\), it follows from the Divergence Theorem in \(B_{\rho}(0)\) that
	\[{}
	\int_{\partial B_{\rho}(0)} \frac{x - y}{\abs{x - y}^{s + 1}} \cdot \frac{x}{\abs{x}} \dif\sigma(x){}
	= \int_{B_{\rho}(0)} \Delta g
	= (d - s - 1) \int_{B_{\rho}(0)} \frac{\dif x}{\abs{x - y}^{s + 1}},
	\]
	which implies \eqref{eq-588}.
	When \(s = d - 1\), the function \(g\) is harmonic in \(\R^{d} \setminus \{y\}\).{}
	An application of the Divergence Theorem on \(B_{\rho}(0) \setminus B_{\eta}(y)\) with \(\eta \to 0\) yields
	\[{}
	\int_{\partial B_{\rho}(0)} \frac{x - y}{\abs{x - y}^{s + 1}} \cdot \frac{x}{\abs{x}} \dif\sigma(x){}
	= 
	\begin{cases}
		\cH^{d - 1}(\partial B_{1}(0)) & \text{if \(y \in {B_{\rho}(0)}\),}\\
		0 & \text{if \(y \not\in \overline{B_{\rho}(0)}\),}
	\end{cases} 
	\]
	which also gives \eqref{eq-588} and completes the proof of the claim.
	
	Combining \eqref{eq-494} and \eqref{eq-588}, we get
	\[{}
	\int_{B_{2r}(0)} \frac{x - y}{\abs{x - y}^{s + 1}} \cdot \frac{x}{\abs{x}} \dif x
	\ge c \, r^{d - s}\chi_{B_{r}(0)}(y),
	\]
	which by \eqref{eq-487} then gives
	\[{}
	\int_{B_{2r}(0)}{V_{s}(x) \cdot \frac{x}{\abs{x}} \dif x}
	\ge c \, r^{d - s} m(B_{r}(0)).
	\]
	Since \(x/\abs{x}\) has zero average, we can subtract any constant vector \(\alpha\) in the integrand in the left-hand side without changing the value of the integral.
	The conclusion readily follows since \(x/\abs{x}\) is a unit vector.
\end{proof}

	When \(s > d - 1\), one has
	\[{}
	\int_{\partial B_{\rho}(0)} \frac{x - y}{\abs{x - y}^{s + 1}} \cdot \frac{x}{\abs{x}} \dif\sigma(x){}
	= 
	\begin{cases}
	\displaystyle (s + 1 - d) \int_{\R^{d} \setminus B_{\rho}(0)} \frac{\dif x}{\abs{x - y}^{s + 1}}	 & \text{if \(y \not\in \overline{B_{\rho}(0)}\),}\\
	\displaystyle (d - s - 1) \int_{B_{\rho}(0)} \frac{\dif x}{\abs{x - y}^{s + 1}}	 & \text{if \(y \in {B_{\rho}(0)}\).}
	\end{cases} 
	\]
	In particular, these values have opposite signs, in contrast with the previous proof.
	We prove in that case the following counterpart of Proposition~\ref{lemmaEstimateDensity} that is enough for our purposes:

\begin{proposition}
	\label{lemmaEstimateDensityNew}
	Let \(d - 1 < s < d\).{}
	For every \(\alpha \in \R^{d}\) and every \(r > 0\),
	\[{}
	\frac{m(B_{r}(a))}{r^{s}}
	\le C' \biggl( \fint_{B_{r}(a)}{\abs{V_{s} - \alpha}}
	+ r^{d - s} \int_{\R^{d} \setminus B_{r}(a)}{ \frac{\abs{V_{s}(x) - \alpha}}{\abs{x - a}^{2d - s}} \dif x } \biggr),
	\]
	for some constant \(C' > 0\) depending on \(s\) and \(d\).
\end{proposition}

As the vector-field \((x - a)/\abs{x - a}\) does not seem to be a convenient choice to test \(V_{s}\) in this range of \(s\), we rely on a different one that is provided by 

\begin{lemma}
	\label{lemmaInversion}
	Let \(d - 1 < s < d\).{}
	For every \(\varphi \in C_{c}^{\infty}(\R^{d})\), there exists a summable smooth vector-field \(\Psi : \R^{d} \to \R^{d}\) such that \(\int_{\R^{d}}{\Psi} = 0\),
	\[{}
	\varphi(x)
	= \int_{\R^{d}}{\frac{x - y}{\abs{x - y}^{s + 1}} \cdot \Psi(y) \dif y}
	\quad \text{for every \(x \in \R^{d}\),}
	\]
	and
	\[{}
	\abs{\Psi(x)} 
	\le \frac{C''}{(1 + \abs{x})^{2d - s}}
	\quad \text{for every \(x \in \R^{d}\),}
	\]
	for some constant \(C'' > 0\) depending on \(s\), \(d\) and \(\varphi\).
\end{lemma}

The existence of \(\Psi\) is well-known and can be obtained as a simple exercise using the Fourier transform and the identity
\[{}
\widehat{\varphi}
= \widehat{\frac{z}{\abs{z}^{s + 1}}} \cdot \widehat{\Psi}
= c \,  \frac{\xi}{\abs{\xi}^{d - s + 1}}  \cdot \widehat{\Psi}
\]
for a constant \(c \in \mathbb{C} \setminus \{0\}\).{}
For the convenience of the reader, we explain below that one can take \(\Psi : \R^{d} \to \R^{d}\) directly given by
\begin{equation}
	\label{eq-715}
	\Psi(x)
	= c' \int_{\R^{d}}{\frac{\nabla \varphi(y)}{\abs{x - y}^{2d - s - 1}} \dif y},
\end{equation}
	with \(c'\in \R \setminus \{0\}\).{}
	 Observe that the integral above is well-defined for \(s > d - 1\).
	
	\resetconstant
	\begin{proof}[Proof of Lemma~\ref{lemmaInversion}]
	That \(\Psi\) in \eqref{eq-715} satisfies the desired integral relation with \(\varphi\) follows from an explicit computation of the convolution between the homogeneous functions \(z/\abs{z}^{\ell}\) and \(1/\abs{z}^{j}\)\,; see Lemma~15.10 and Exercise~15.4 in \cite{Ponce:2016} for dimensions \(d \ge 2\) and \(d = 1\), respectively. 
	We also note that \(\Psi\) is bounded and smooth.
	Using that \(\int_{\R^{d}}{\nabla \varphi} = 0\), one writes
	\[{}
	\Psi(x)
	= c' \int_{\R^{d}}{\nabla \varphi(y) \biggl(\frac{1}{\abs{x - y}^{2d - s - 1}} - \frac{1}{\abs{x}^{2d - s - 1}} \biggr) \dif y}.
	\]
	Take \(R > 0\) such that \(\supp{\varphi} \subset B_{R}(0)\).
	For \(x \in \R^{d} \setminus B_{2R}(0)\), an application of the Mean Value Theorem gives 
	\[{}
	\abs{\Psi(x)}
	\le \int_{B_{R}(0)}{\abs{\nabla \varphi(y)} \biggabs{\frac{1}{\abs{x - y}^{2d - s - 1}} - \frac{1}{\abs{x}^{2d - s - 1}}} \dif y}
	\le \biggl( \int_{\R^{d}}{\abs{\nabla \varphi}} \biggr) \frac{\C}{\abs{x}^{2d - s}}.
	\]
	The estimate of \(\abs{\Psi(x)}\) in \(\R^{d}\) then follows from the boundedness of \(\Psi\).{}
	In particular, \(\Psi\) is summable for \(s < d\) and one then verifies that
	\(
	\int_{\R^{d}}{\Psi} = 0.
	\)
\end{proof}

\resetconstant
\begin{proof}[Proof of Proposition~\ref{lemmaEstimateDensityNew}]
	We may assume that \(a = 0\). 
	Fix a nonnegative function \(\varphi \in C_{c}^{\infty}(\R^{d})\) such that \(\varphi \ge 1\) in \(B_{1}(0)\) and let \(\Psi : \R^{d} \to \R^{d}\) be the vector-field given by Lemma~\ref{lemmaInversion}.
	For \(r > 0\),  one has by scaling that
	\[{}
	\varphi\Bigl(\frac{x}{r}\Bigr)
	= \int_{\R^{d}}{\frac{x - y}{\abs{x - y}^{s + 1}} \cdot \Psi_{r}(y) \dif y}
	\quad \text{for every \(x \in \R^{d}\),}
	\]
	where \(\Psi_{r}(x) \vcentcolon= \frac{1}{r^{d - s}} \Psi(\frac{x}{r})\).{}
		Since \(\varphi \ge 1\) on \(B_{1}(0)\) and \(m\) is a nonnegative measure,{}
	we have
	\[{}
	m(B_{r}(0))
	\le \int_{\R^{d}}{\varphi\Bigl(\frac{x}{r}\Bigr) \dif m(x)}
	= \int_{\R^{d}}{\biggl( \int_{\R^{d}}{\frac{x - y}{\abs{x - y}^{s + 1}} \cdot \Psi_{r}(y) \dif y} \biggr) \dif m(x)}.
	\]
	Then, by Fubini's theorem,
	\[{}
	m(B_{r}(0))
	\le - \int_{\R^{d}} V_{s} \cdot \Psi_{r} \,.
	\]
	Since \(\int_{\R^{d}}{\Psi_{r}} = 0\), for every \(\alpha \in \R^{d}\) we get
	\begin{equation}
		\label{eq-697}
	m(B_{r}(0))
	\le - \int_{\R^{d}} (V_{s} - \alpha) \cdot \Psi_{r} 
	\le \int_{\R^{d}} \abs{V_{s} - \alpha} \abs{\Psi_{r}}.
	\end{equation}
	The pointwise estimate satisfied by \(\Psi\) gives
	\begin{equation}
		\label{eq-708}
	\abs{\Psi_{r}(x)} = \frac{1}{r^{d - s}} \Bigabs{\Psi\Bigl(\frac{x}{r}\Bigr)} 
	\le \C \Bigl( \frac{1}{r^{d - s}} \chi_{B_{r}(0)}(x) + \frac{r^{d}}{\abs{x}^{2d - s}} \chi_{\R^{d} \setminus B_{r}(0)}(x) \Bigr).
	\end{equation}
	The conclusion follows by inserting \eqref{eq-708} in \eqref{eq-697}.
\end{proof}

To handle the additional term that appears in Proposition~\ref{lemmaEstimateDensityNew}, 
compared to Proposition~\ref{lemmaEstimateDensity}, we need the following

\begin{lemma}
	\label{lemmaComputation}
	If \(f \in L^{1}(\R^{d})\) is such that
	\[{}
	\lim_{r \to 0}{\fint_{B_{r}(a)}{f}}
	= 0,
	\]
	then, for every \(\beta > 0\),{}
	\[{}
	\lim_{r \to 0}{r^{\beta} \int_{\R^{d} \setminus B_{r}(a)} \frac{f(x)}{\abs{x - a}^{d + \beta}} \dif x}
	= 0.
	\]
\end{lemma}

\resetconstant
\begin{proof}
	We first prove that, for every \(r > 0\),{}
	\begin{equation}
	\label{eq-738}
	\int_{\R^{d} \setminus B_{r}(a)}{\frac{f(x)}{\abs{x - a}^{d + \beta}} \dif x}
	= \Cl{cte-803} \int_{r}^{\infty} \frac{1}{\rho^{1 + \beta}} \biggl( \fint_{B_{\rho}(a)}{f} \biggr) \dif\rho -  \frac{\Cl{cte-803-bis}}{r^{\beta}}\fint_{B_{r}(a)}{f},
	\end{equation}
	for some constants \(\Cr{cte-803}, \Cr{cte-803-bis} > 0\).
	Indeed, by the integration formula in polar coordinates,
	\[{}
	\int_{\R^{d} \setminus B_{r}(a)}{\frac{f(x)}{\abs{x - a}^{d + \beta}} \dif x}
	 = \int_{r}^{\infty}\biggl( \int_{\partial B_{\rho}(a)} \frac{f(x)}{\abs{x - a}^{d + \beta}} \dif\sigma(x)  \biggr) \dif\rho{}
	 = \int_{r}^{\infty} \frac{1}{\rho^{d + \beta}} \biggl( \frac{\dif}{\dif\rho} \int_{B_{\rho}(a)} f \biggr) \dif\rho.
	\]
	One then gets \eqref{eq-738} by integration by parts.
	To conclude, it suffices to apply Lemma~\ref{lemmaLimit} with \(h(t) = \fint_{B_{t}(a)}{f}\) for \(t > 0\).
\end{proof}

\resetconstant
\begin{proof}[Proof of Theorem~\ref{theoremRieszLebesgue}] ``\(\Longrightarrow\)''.
	Let \(a\) be a Lebesgue point of \(V_{s}\).{}
	We first show that the limit \eqref{eq-328} holds.
	To this end, denote by \(\alpha\) the precise representative of \(V_{s}\) at \(a\).{}
	When \(0 < s \le d - 1\), we deduce from Proposition~\ref{lemmaEstimateDensity} that
	\[
	\limsup_{r \to 0}{\frac{m(B_{r}(a))}{r^{s}}}
	\le C \lim_{r \to 0}{\fint_{B_{2r}(a)}{\abs{V_{s} - \alpha}}}
	= 0,
	\]
	which implies \eqref{eq-328}.
	When \(d - 1 < s < d\)\,, we apply Proposition~\ref{lemmaEstimateDensityNew}.
	In this case, Lemma~\ref{lemmaComputation} with \(f = |V_{s} - \alpha|\) and \(\beta = d - s\) handles the additional term in the estimate as \(r \to 0\) and we get \eqref{eq-328} as before.	
	
	We next recall that
	\[{}
	\alpha = \lim_{r \to 0}{\fint_{B_{r}(a)}{V_{s}}}.
	\]
	Moreover, by estimate \eqref{eq-334} in the proof of Theorem~\ref{theoremRieszLebesgue}, we have
	\[
	\biggabs{\fint_{B_{r}(a)} V_{s} - \int_{\R^{d} \setminus B_{2r}(a)} \frac{a - y}{\abs{a - y}^{s + 1}} \dif m(y)}
	\le \C \biggl( r \int_{2r}^{\infty} \frac{m(B_{t}(a))}{t^{s + 2}} \dif t + \frac{m(B_{2r}(a))}{r^{s}} \biggr){}.
	\]
	As \(r \to 0\), using \eqref{eq-328} we deduce from Lemma~\ref{lemmaLimit} that the right-hand side converges to zero and then
	\[{}
	\alpha = \lim_{r \to 0}{\fint_{B_{r}(a)}{V_{s}}} = \lim_{r \to 0}{\int_{\R^{d} \setminus B_{2r}(a)} \frac{a - y}{\abs{a - y}^{s + 1}} \dif m(y)}.
	\]
	Thus, by definition, \(\alpha\) is the principal value of \(V_{s}\) at \(a\).
\end{proof}

\section{The capacity $\kappa_{s}$ and proof of Theorem~\ref{theoremCapacity}}
\label{sectionCapacity}

	We begin by showing that \(\kappa_{s}\) is subadditive:
	
\begin{proposition}
	\label{propositionSubadditivity}
	For all compact subsets \(E_{1}, E_{2} \subset \R^{d}\),{}
	\[{}
	\kappa_{s}(E_{1} \cup E_{2})
	\le \kappa_{s}(E_{1}) + \kappa_{s}(E_{2}).
	\]
\end{proposition}

The capacity \(\kappa_{s}\) is equivalent to another one that was known to be semi-additive, that is, the estimate above is verified with a constant \(C \ge 1\); see p.~3643 in \cite{RuizdeVilla-Tolsa} and Section~2 in \cite{Prat:2012}.

Before proving the proposition, we start with the following observation:
	If \(\mu\) is a Borel measure in \(\R^{d}\) that is admissible in the definition of \(\kappa_{s}(E)\) and if \(F \subset E\) is a compact subset, then, for any Borel subset \(A \subset F\), the Borel measure \(\mu\lfloor_{A}\) defined by
	\[{}
	\mu\lfloor_{A}(B) = \mu(A \cap B)
	\]
	is admissible for \(\kappa_{s}(F)\).{}
	Indeed, it suffices to verify the \(L^{2}\) boundedness of the maximal Riesz transform.
	Since \(R_{s}^{*}\) is bounded in \(L^{2}(\mu)\), we can estimate 
	\[{}
	\int_{\R^{d}}{\abs{R_{s}^{*}(f\mu\lfloor_{A})}^{2} \dif \mu\lfloor_{A}}
	 \le \int_{\R^{d}}{\abs{R_{s}^{*}(f\chi_{A} \mu)}^{2} \dif \mu}
	 \le \int_{\R^{d}}{\abs{f\chi_{A}}^{2} \dif \mu}
	= \int_{\R^{d}}{\abs{f}^{2} \dif \mu\lfloor_{A}},
	\]
	which justifies our assertion.

\begin{proof}[Proof of Proposition~\ref{propositionSubadditivity}]
	Let \(\mu\) be a nonnegative Borel measure supported in \(E_{1} \cup E_{2}\) that satisfies (a) and (b) 
	in the definition of $\kappa_s$ for this set. 
	By the computation above, the measure \(\mu\lfloor_{E_{1}}\) is admissible for \(\kappa_{s}(E_{1})\) and \(\mu\lfloor_{E_{2} \setminus E_{1}}\) is admissible for \(\kappa_{s}(E_{2})\).{}
	Thus,
	\[{}
	\mu(\R^{d})
	= \mu(E_{1} \cup E_{2}) 
	= \mu(E_{1}) + \mu(E_{2} \setminus E_{1}) 
	\le \kappa_{s}(E_{1}) + \kappa_{s}(E_{2})
	\] 
	and it suffices to take the supremum with respect to \(\mu\).
\end{proof}

It follows readily from (a) in the definition of \(\kappa_{s}\) that 
\(\kappa_{s}(E) = 0\) for every compact set \(E\) with \(\cH^{s}(E) < \infty\).{}
Dimension \(s\) is critical for \(\kappa_{s}\)\,:{}

\begin{proposition}
	\label{propositionCapacityHausdorffDimension}	
	If \(E \subset \R^{d}\) is compact and \(\kappa_{s}(E) = 0\), then \(\dim_{\mathcal{H}}{(E)} \le s\).
\end{proposition}

\resetconstant
\begin{proof}
	It suffices to prove that if \(E \subset \R^{d}\) is a compact set with \(\cH^{t}(E) > 0\) for some \(t > s\), then \(\kappa_{s}(E) > 0\).{}
	By Frostman's lemma, there exists a nontrivial finite nonnegative Borel measure \(\mu\) supported by \(E\) such that
	\begin{equation}
	\label{eq-546}
	\mu(B_{r}(x))
	\le r^{t}
	\quad \text{for every \(x \in \R^{d}\) and \(r > 0\).}
	\end{equation}
	Hence, \(\mu \le \C \cH^{t}_{\infty}\)\,, where \(\cH^{t}_{\infty}\) is the Hausdorff content of dimension \(t\).{}
	For every set \(A \subset \R^{d}\) with \(\sigma\)-finite \(\cH^{s}\) measure, we have \(\cH^{t}_{\infty}(A) = 0\). 
	Then, \(\mu(A) = 0\) and the first requirement in the definition of \(\kappa_{s}(E)\) is satisfied.
	
	Next, 
	for every \(f \in (L^{1} \cap L^{2})(\mu)\) and every \(x \in \R^{d}\),
	\[{}
	\abs{R_{s}^{*}(f\mu)(x)}
	\le \int_{\R^{d}} \frac{\abs{f(y)}}{\abs{x - y}^{s}} \dif \mu(y).
	\]
	Then, by Young's inequality,
	\[{}
	\norm{R_{s}^{*}(f\mu)}_{L^{2}(\mu)}
	\le \biggl(\int_{\R^{d}} \frac{\dif \mu(z)}{\abs{z}^{s}} \biggr) \norm{f}_{L^{2}(\mu)}.
	\]
	By Cavalieri's principle,
	\[{}
	\int_{\R^{d}}{\frac{\dif \mu(z)}{\abs{z}^{s}}} 
	= s \int_{0}^{\infty} \frac{\mu(B_{r}(0))}{r^{s + 1}} \dif r.
	\]
	Since \(\mu\) is finite and satisfies \eqref{eq-546} with exponent \(t > s\),{}
	\[{}
	\int_{0}^{\infty} \frac{\mu(B_{r}(0))}{r^{s + 1}} \dif r
	\le \int_{0}^{1} \frac{r^{t}}{r^{s + 1}} \dif r + \int_{1}^{\infty} \frac{\mu(\R^{d})}{r^{s + 1}} \dif r
	< \infty.
	\]
	Therefore,
	\[{}
	\norm{R_{s}^{*}(f\mu)}_{L^{2}(\mu)}
	\le \Cl{cte-423} \norm{f}_{L^{2}(\mu)},
	\]
	which implies that \(\mu/\Cr{cte-423}\) is a nontrivial admissible measure in the definition of \(\kappa_{s}(E)\). 
	In particular, \(\kappa_{s}(E) > 0\). 
\end{proof}

The proof of Theorem~\ref{theoremCapacity} relies on Theorem~1.6 of \cite{Mattila-Verdera} concerning the existence of principal values for the Riesz transform.
We recall that from \cite{Mattila-Verdera} one knows that, for every finite nonnegative Borel measure \(\mu\) in \(\R^{d}\) such that
\begin{equation}
\label{eq-692}
	\lim_{r \to 0}{\frac{\mu(B_{r}(x))}{r^{s}}}
	= 0
	\quad \text{for \(\mu\)-almost every \(x \in \R^{d}\)}
\end{equation}
and
\begin{equation}
\label{eq-698}
	\norm{R_{s}^{*}(f\mu)}_{L^{2}(\mu)}
	\le C \norm{f}_{L^{2}(\mu)}
	\quad \text{for every \(f \in (L^{1} \cap L^{2})(\mu)\)},
\end{equation}
the principal value of 
\begin{equation}
\label{eq-705}
	x \longmapsto \int_{\R^{d}} \frac{x - y}{\abs{x - y}^{s + 1}} \dif\mu(y)
\end{equation}
	exists \(\mu\)-almost everywhere.

	We observe that in the statement of Theorem~1.6 in \cite{Mattila-Verdera} one also assumes
	\[{}
	\mu(B_{r}(x))
	\le C' r^{s} \quad \text{for every \(x \in \R^{d}\) and \(r > 0\),}
	\]
	but such a property is a consequence of the uniform boundedness of the Riesz transform given by \eqref{eq-698} and the 
	fact that, by \eqref{eq-692}, the measure \(\mu\) cannot charge points; see Proposition~1.4 in Part~III of \cite{David}.

	We also need a standard property from Measure Theory:
	
\begin{proposition}
	\label{propositionMattilaVerdera}
	If \(\mu\) is a finite nonnegative Borel measure in \(\R^{d}\) such that \(\mu(A) = 0\) for every Borel set \(A \subset \R^{d}\) with \(\sigma\)-finite \(\cH^{s}\) measure, then \eqref{eq-692} holds.
\end{proposition}

\begin{proof}
	Since \(\mu\) does not charge sets with finite \(\mathcal{H}^{s}\) measure, by Proposition~3.2 in \cite{Ponce} for every \(\epsilon > 0\) one finds a compact set \(K \subset \R^{d}\) with \(\mu(\R^{d} \setminus K) \le \epsilon\) such that, for every \(c > 0\), there exists \(\delta > 0\) with
	\begin{equation*}
	\mu\lfloor_{K}(B_{r}(x))
	\le c r^{s} \quad \text{for every \(x \in \R^{d}\) and \(0 < r \le \delta\).}
	\end{equation*}
	Thus,
\begin{equation}
\label{eq-1009}
	\lim_{r \to 0}{\frac{\mu\lfloor_{K}(B_{r}(x))}{r^{s}}}
	= 0
	\quad \text{for every \(x \in \R^{d}\).}
\end{equation}
	On the other hand, by the Besicovitch Differentiation Theorem,
\begin{equation}
\label{eq-1016}
	\lim_{r \to 0}{\frac{\mu\lfloor_{K}(B_{r}(x))}{\mu(B_{r}(x))}}
	= 1
	\quad \text{for \(\mu\)-almost every \(x \in K\).}
\end{equation}
	Combining \eqref{eq-1009} and \eqref{eq-1016}, one gets the limit in \eqref{eq-692} for \(\mu\)-almost every \(x \in K\).
	Since \(\mu(\R^{d} \setminus K) \le \epsilon\) and \(\epsilon > 0\) is arbitrary, the conclusion follows.
\end{proof}
	
\begin{proof}[Proof of Theorem~\ref{theoremCapacity}]
	Let \(\mu\) be a finite nonnegative Borel measure that is admissible for \(\kappa_{s}(E)\), where \(E\) is any compact subset of \(\mathcal{E}(V_{s})\).{}
	By Theorem~\ref{theoremRieszLebesgue}, we can write this set as \(E = E_{1} \cup E_{2}\), where
	\[{}
	E_{1} 
	\vcentcolon= \biggl\{ x \in E : \limsup_{r \to 0}{\frac{\mu(B_{r}(x))}{r^{s}}} > 0 \biggr\}
	\]
	and
	\[{}
	E_{2} 
	\vcentcolon= \bigl\{ x \in E : \mathrm{p.v.}\,V_{s}(x) \text{ does not exist} \bigr\}.
	\]
	On one hand, since \(E_{1}\) is \(\sigma\)-finite for \(\mathcal{H}^{s}\) and \(\mu\) does not charge those sets,
	\(
	\mu(E_{1}) = 0.
	\)
	On the other hand, from Proposition~\ref{propositionMattilaVerdera} above and Theorem~1.6 in \cite{Mattila-Verdera} we know that the principal value of \(V_{s}\) exists \(\mu\)-almost everywhere in \(\R^{d}\).
	Hence,
	\(
	\mu(E_{2}) = 0.
	\)
	We thus have 
	\[{}
	\mu(\R^{d}) 
	= \mu(E) 
	\le \mu(E_{1}) + \mu(E_{2})
	= 0,
	\] 
	for every measure \(\mu\) that is admissible for \(\kappa_{s}(E)\).
	Therefore, \(\kappa_{s}(E) = 0\).
\end{proof}

	As a final observation, \eqref{eqHausdorffDimension} can also be deduced as a consequence of
	Theorem~\ref{theoremCapacity}.
	Indeed, if we had \(\dim {\mathcal{E}(V_{s})} > s\), then for any \(s < t < \dim{\mathcal{E}(V_s)}\) one could find a
	compact subset \(F \subset \mathcal{E}(V_{s})\) such that \(\mathcal{H}^{t}(F) > 0\); see Theorem~8.13 in \cite{Mattila}.
	Such a property would then contradict Proposition~\ref{propositionCapacityHausdorffDimension} above.
	

\section{Planar Cantor sets and the Cauchy integral}	
\label{sectionCantor}

We investigate in this section the existence of principal values of the Cauchy integral for the intrinsic probability measure on Cantor sets in \(\R^{2}\) which we identify with the complex plane \(\CC\).
We first review the definition of the Cantor sets we shall be dealing with. 
Take a sequence $\Lambda = (\lambda_n)_{n \in \N_{*}}$  with $0 <\lambda_n < 1/2$ for every $n$, where \(\N_{*} = \N \setminus \{0\}\).
Start with the square $S_{0} = [0, 1]^{2}$ with sides parallel to the coordinate axes and side length $\sigma_{0} \vcentcolon= 1$.
Take $4$ squares of side length $\sigma_{1} \vcentcolon= \lambda_1$ inside \(S_{0}\) with sides parallel to the coordinate axes,  each containing a vertex of $S_{0}$. Repeat the operation in each of the $4$ squares obtained with $\lambda_2$ as the compression factor. 
We then obtain $16$ squares of side length $\sigma_{2} \vcentcolon= \lambda_1 \lambda_2$.
Repeating inductively the process we get at the $n$th generation $4^n$ squares of side length 
\[{}
\sigma_n{}
\vcentcolon =\lambda_1\cdots\lambda_n.
\]
Denote these squares by $S_{n, \alpha}$ with \(\alpha \in \{1, 2, \dots, 4^n\}\). 
The Cantor set associated with the sequence \(\Lambda\) is 
\begin{equation}\label{ca}
E(\Lambda) \vcentcolon= \bigcap_{n=1}^\infty \bigcup_{\alpha = 1}^{4^n} S_{n, \alpha}.
\end{equation}
The intrinsic probability measure $\mu$ on $E(\Lambda)$ that we consider has the property 
\[{}
\mu(S_{n, \alpha}) = 1/4^n
\quad \text{for every $n$ and $\alpha$.}
\]

\begin{example}
When $\lambda_n=1/4$ for every $n$, \(E(\Lambda)\) is called the Garnett-Ivanov set.
In this case, the intrinsic measure $\mu$ is a positive multiple of the one-dimensional Hausdorff measure on $E(\Lambda)$.{}
This is the simplest example of a compact set having positive length but zero analytic capacity. 
Another interesting family of Cantor sets \(E(\Lambda)\) arises setting, for a given $0 < \beta < \infty$ and for every \(n\) sufficiently large,
\begin{equation}\label{cabe}
\lambda_n{}
= \frac{1}{4}\Bigl(1+\frac{\beta}{n} \Bigr).
\end{equation}
In this case, one has $\sigma_n \sim {n^\beta}/{4^n}$.
\end{example}

When the intrinsic measure \(\mu\) has zero linear density, that is
\[{}
\frac{\mu(S_{n, \alpha})}{\sigma_{n}}
= \frac{1}{4^{n}\sigma_{n}}
\to 0
\quad \text{as \(n \to \infty\),}
\]
assumption \eqref{eq-328} of Theorem~\ref{theoremRieszLebesgue} with \(s = 1\) is satisfied by \(\mu\).{}
Hence, \(a \in E(\Lambda)\) is a Lebesgue point of \(C(\mu)\) if and only if \(C(\mu)\) has a principal value at \(a\).{}
Using the result concerning \(\mu\)-almost everywhere existence of principal values from \cite{Mattila-Verdera}, one has the following

\begin{proposition}
	If 
	\begin{equation}\label{eq1096}
	\sum_{n=1}^\infty \frac{1}{(4^n \sigma_n)^2} < \infty,
	\end{equation}
	then \(C(\mu)\) has a principal value for \(\mu\)-almost every point in \(E(\Lambda)\).
\end{proposition}

\begin{proof}
	By Theorem~1 in \cite{Mateu-Tolsa-Verdera}, assumption \eqref{eq1096} is equivalent to the fact that \(E(\Lambda)\) has positive analytic capacity, which is also equivalent to \(L^{2}(\mu)\) boundedness of \(C(\mu)\).{}
	The conclusion then readily follows from Theorem~1.6 in \cite{Mattila-Verdera}.
\end{proof}

Since the Garnett-Ivanov set \(E(\Lambda)\) has positive linear density, by Theorem~\ref{theoremRieszLebesgue} the Cauchy integral \(C(\mu)\) has no Lebesgue points on \(E(\Lambda)\).{}
In fact, failure of Theorem~\ref{theoremRieszLebesgue} in this case is twofold and also concerns the non-existence of principal values of \(C(\mu)\) on \(E(\Lambda)\):

\begin{proposition}
\label{propositionGI}
If \(E(\Lambda)\) is the Garnett-Ivanov set, then \(C(\mu)\) does not have a principal value at any point of \(E(\Lambda)\).
\end{proposition}

\begin{proof}
Let $a \in E(\Lambda)$ and let $Q_{n}$ denote the unique square of generation $n$ that contains $a$.
One has $E(\Lambda) \setminus B_{\sqrt{2}/4^n}(a)= E(\Lambda)\setminus Q_{n}$.
Hence, if the principal value of the Cauchy integral exists
at $a$, then we get that the limit
\begin{equation*}\label{}
\lim_{n\to \infty} \int_{E(\Lambda) \setminus Q_{n}} \frac{\dif\mu(z)}{z-a} 
\end{equation*}
also exists. 
In particular,
\begin{equation*}\label{}
\lim_{n\to \infty} \int_{Q_{n}\setminus Q_{n+1}} \frac{\dif\mu(z)}{z-a} = 0.
\end{equation*}
Making  a dilation of factor $4^{n} = 1/\sigma_{n}$  followed by an appropriate translation we transform $Q_{n}$ into the initial square $S_{0}$\,, the point $a$ is mapped to a point $b_n$ in $E(\Lambda)$, and $Q_{n + 1}$ is mapped into the square $\widetilde Q_{1}$ of first generation that contains $b_n$. 
Thus, since \(\mu(Q_{n}) = \sigma_{n} = 1/4^{n}\),
\begin{equation*}
\int_{Q_{n} \setminus Q_{n + 1}} \frac{\dif\mu(z)}{z-a} 
= \frac{1}{4^{n}\sigma_{n}} \int_{S_{0} \setminus \widetilde Q_{1}} \frac{\dif\mu(z)}{z-b_n}
= \int_{S_{0} \setminus \widetilde Q_{1}} \frac{\dif\mu(z)}{z-b_n}.
\end{equation*}
To reach a contradiction, we only have to check that this last integral is bounded away from $0$. 
It is more convenient to look at the conjugate kernel, 
since 
$$\frac{1}{\overline{z-b_n}}= \frac{z-b_n}{|z-b_n|^2}.$$ 

Let us assume for simplicity that \(Q_{n + 1}\) is the upper-right square inside \(Q_{n}\).{}
Then, \(\widetilde Q_{1}\) has the same relative position inside \(S_{0}\).
One verifies geometrically in this case that if $\vec{u} \vcentcolon = (-1, -1)$, then the half-space $\bigl\{z \in \R^{2} : \langle z-b_n, \vec{u} \rangle > 0\bigr\}$ contains $S_{0} \setminus \widetilde Q_{1}$, where the angle brackets mean scalar product in the plane.
Therefore, by linearity of the integral,
\begin{equation*}\label{}
\Bigl\langle \int_{S_{0} \setminus \widetilde Q_{1}} \frac{\dif\mu(z)}{\overline{z-b_n}}, \vec{u} \Bigr\rangle 
\ge \delta >0,
\end{equation*}
for some \(\delta > 0\) independent of \(n\).
A similar argument holds when \(Q_{n + 1}\) is any other square of generation \(n + 1\) in \(Q_{n}\)\,, by choosing \(\vec{u}\) accordingly.
We then deduce that \(E(\Lambda)\) has no point where \(C(\mu)\) has a principal value.
\end{proof}

When \eqref{eq1096} fails, one has the following property
concerning the non-existence of principal values of \(C(\mu)\) on \(E(\Lambda)\):{}

\begin{proposition}
\label{propositionMeasureP}
If 
\begin{equation}\label{cappos}
\sum_{n=1}^\infty \frac{1}{(4^n \sigma_n)^2} = \infty,
\end{equation}
then \(C(\mu)\) does not have a principal value for \(\mu\)-almost every point in \(E(\Lambda)\).
\end{proposition}

\resetconstant
\begin{proof}
We rely on a strong result of Tolsa \cite{Tolsa:book}*{Section~8.15}, who proved in a much more general context that the set
\begin{equation}
	\label{eqSetP}
P(\Lambda) \vcentcolon= \bigl\{z \in E(\Lambda) : \mathrm{p.v.}\, C(\mu)(z)\ \text{exists} \bigr\}
\end{equation}
is the union of a set of zero $\mu$ measure and a sequence of subsets $F_l$ in $E(\Lambda)$ such that $\mu\lfloor_{F_l}$ has finite Menger curvature, that is
\begin{equation}\label{mc}
\iiint_{F_{l}^3} \frac{\dif\mu(z) \dif\mu(w) \dif\mu(\zeta)}{R^{2}(z, w, \zeta)} < \infty,
\end{equation}
where $R(z, w, \zeta)$ is the radius of the circle through $z$, $w$ and $\zeta$\,; when two of these point coincide, $R(z, w, \zeta)= \infty$.{}

For each \(l \in \N\), let \(a \in E(\Lambda)\) be a point of density of $F_{l}$ for \(\mu\) with respect to squares, that is
\begin{equation*}
\lim_{r \to 0}{\frac{\mu(Q_{r}(a) \cap F_{l})}{\mu(Q_{r}(a))}}
= 1.
\end{equation*}
Here, \(Q_{r}(a)\) denotes the square centred at \(a\) with side length \(2r\).
The set of points of density coincides, modulo a set of vanishing $\mu$ measure, with the set of points $a$ satisfying
\begin{equation}
	\label{eqDensity2}
\lim_{n \to \infty}{\frac{\mu(Q_n \cap F_{l})}{\mu(Q_n)}}
= 1,
\end{equation}
where $Q_n$ is the square of generation $n$ that contains $a$. 
This follows by considering the maximal dyadic function
\begin{equation*}
M(f)(a)=\sup_{n}{\frac{1}{\mu(Q_n)} \int_{Q_n} |f| \dif\mu} \quad \text{for \(f\in L^1(\mu)\).}
\end{equation*}
Indeed, one first shows that this maximal function satisfies a weak type inequality, from which the argument can then be completed in the standard way.

Inside $Q_{n}$\,, let us take two squares of generation $n+1$ different from $Q_{n+1}$\,, say $S_{n+1, 1}$ and $S_{n+1, 2}$.
For every \(n \in \N_{1} \vcentcolon= \bigl\{ n \in \N : \lambda_{n} \le 1/3 \bigr\}\), the sets $Q_{n+1}$\,, $S_{n+1, 1}$ and $S_{n+1, 2}$ are uniformly away from each other relatively to \(\sigma_{n}\)\,, that is, if \(w\) and \(\zeta\) are two points belonging to distinct sets, then
\[{}
|w - \zeta| 
\ge \sigma_{n}(1 - 2\lambda_{n})
\ge \frac{\sigma_{n}}{3}.
\]
We then have for each \(n \in \N_{1}\)\,, 
\[{}
R(a, w, \zeta) \sim \sigma_{n}
\quad \text{for every \(w \in S_{n+1, 1}\) and \(\zeta \in S_{n+1, 2}\).}
\]
As the sets \(S_{n+1, j}\) are pairwise disjoint for \(j \in \{1, 2\}\), for every \(n \in \N_{1}\) we can estimate
\begin{equation}
\label{eq1128}	
\begin{split}
\iint_{F_{l}^2} \frac{\dif\mu(w) \dif\mu(\zeta)}{R^{2}(a, w, \zeta)}
& \ge \sum_{n \in \N_{1}} \int_{S_{n+1, 1} \cap F_{l}} \int_{S_{n+1, 2} \cap F_{l}}
\frac{\dif\mu(w) \dif\mu(\zeta)}{R^{2}(a, w, \zeta)}\\
& \ge \Cl{cte-1185} \sum_{n \in \N_{1}} \frac{\mu(S_{n+1, 1} \cap F_{l})\mu(S_{n+1, 2} \cap F_{l})}{\sigma_{n}^{2}}.
\end{split}
\end{equation}

Let us assume that \(a\) satisfies \eqref{eqDensity2}.
Then, there exists \(N \in \N\) such that for \(n \ge N\) and  \(j \in \{1, 2\}\),
$$
\mu(S_{n+1, j} \cap F_{l}) 
\ge \frac{1}{2} \mu(S_{n+1, j})
= \frac{1}{2} \frac{1}{4^{n+1}}.
$$
It thus follows from \eqref{eq1128} that
\begin{equation}
\label{eq1230}
\iint_{F_{l}^2} \frac{\dif\mu(w) \dif\mu(\zeta)}{R^{2}(a, w, \zeta)}
\ge \frac{\Cr{cte-1185}}{4^{3}} \sum_{\substack{n \in \N_{1}\\ n \ge N}} \frac{1}{(4^{n} \sigma_{n})^{2}}.
\end{equation}
A simple computation shows that
\[{}
\sum_{n \in \N \setminus \N_{1}}{\frac{1}{(4^{n} \sigma_{n})^{2}}}
\le 3 \biggl(1 + \sum_{n \in \N_{1}}{\frac{1}{(4^{n} \sigma_{n})^{2}}} \biggr).
\]
From \eqref{cappos}, we then deduce that the series in \eqref{eq1230} diverges, whence so does the integral in the left-hand side.
Since this holds for $\mu$-almost all density point \(a\) of $F_{l}$\,, we deduce from \eqref{mc} that \(\mu(F_{l}) = 0\) for every \(l\).{}
Therefore, \(\mu(P(\Lambda)) = 0\).	
\end{proof}

When \(\Lambda\) is given by \eqref{cabe} for \(n\) large, one has that \eqref{cappos} holds provided that \(\beta \le 1/2\) and then \(C(\mu)\) does not have a principal value for \(\mu\)-almost every point in \(E(\Lambda)\).
We prove nevertheless in the next section that, when \(E(\Lambda)\) has linear density zero, \(C(\mu)\) has a principal value somewhere in \(E(\Lambda)\):

\begin{theorem}
	\label{theoremCantorPV}
	Let \(\Lambda\) be such that \(\sum\limits_{n=1}^{\infty}{\abs{\lambda_{n} - \lambda_{n+1}}} < \infty\).{}
	If the sequence \((4^{n}\sigma_{n})_{n \in \N}\) is nondecreasing and
	\begin{equation}
		\label{eqDensityLambda}
	\lim_{n \to \infty}{4^{n}\sigma_{n}} = \infty,
	\end{equation}
	then there exists \(a \in E(\Lambda)\) such that \(C(\mu)\) has a principal value at \(a\). 
	Equivalently, there exists a Lebesgue point of \(C(\mu)\) in \(E(\Lambda)\).
\end{theorem}

Observe that \((4^{n}\sigma_{n})_{n \in \N}\) is nondecreasing if and only if \(\lambda_{n} \ge 1/4\) for every \(n \in \N_{*}\).
We recall that a point in the Cantor set $E(\Lambda)$ is uniquely determined by a decreasing sequence of squares $Q_n$, where $Q_n$ is the square of generation $n$ that contains the point. 
The point $a$ that we use in Theorem~\ref{theoremCantorPV} is obtained by requiring that it belongs to the upper-right square (position I) of first generation, the lower-left square  (position III) of second generation and so on, alternating positions I and III at each generation. 
We call \(a\) the \emph{alternating point} of \(E(\Lambda)\).
It is chosen so that maximal cancellation is achieved in the computation of the principal value of \(C(\mu)\) at \(a\). 
We are indebted to P.~Mattila for the suggestion of considering such a point.

Theorem~\ref{theoremCantorPV} provides one with a general condition which ensures that \(P(\Lambda)\) given by \eqref{eqSetP} is non-empty.
A challenging problem is

\begin{openproblem}
	Given a sequence \(\Lambda\) satisfying \eqref{eqDensityLambda}, estimate the Hausdorff dimension of the set \(P(\Lambda)\).
\end{openproblem}

As Theorem~\ref{theoremCantorPV} concerns the intrinsic measure \(\mu\), another interesting question is

\begin{openproblem}
Does there exist a compactly supported positive Borel measure $\nu$ in \(\CC\) with
\[{}
\lim_{r \to 0}{\frac{\nu(B_{r}(a))}{r}} = 0
\quad \text{for every \(a \in \CC\)}
\] 
and such that  $C(\nu)$ does not have a principal value at any point of \(\supp{\nu}\)\,?	
In case the answer is affirmative, then \(\supp{\nu}\) would coincide with the set of non-Lebesgue points of $C(\nu)$.  
\end{openproblem}


\section{Proof of Theorem~\ref{theoremCantorPV}}
\label{sectionCantorProof}

We proceed by assuming that there exists \(\delta > 0\) such that
\begin{equation}
	\label{eq1253}
	\lambda_{n}
	\le \frac{1}{2} - \delta{}
	\quad \text{for every \(n \in \N_{*}\).}
\end{equation}
Since the sequence \(\Lambda\) converges and \(\lambda_{n} < 1/2\) for every \(n\), when \eqref{eq1253} fails one must have \(\lambda_{n} \to 1/2\).{}
This case can be easily handled by a separate argument, with a stronger conclusion, and will be briefly explained at the end of the section.

Given \(a \in E(\Lambda)\), to prove the existence of
\begin{equation}
	\label{eq1249}
\lim_{r \to 0}{\int_{\CC \setminus B_{r}(a)}}\frac{\dif\mu(z)}{z - a},
\end{equation}
it suffices to prove that the following limit exists:
\begin{equation}
	\label{eq1249bis}
\lim_{n \to \infty}{\int_{\CC \setminus Q_n}}\frac{\dif\mu(z)}{z - a},
\end{equation}
where \(Q_{n}\) is the square of generation \(n\) that contains \(a\).
To see this, one can argue as follows: Given $r > 0$, let $n$ be the largest nonnegative integer such that $B_{r}(a) \cap E(\Lambda) \subset Q_{n}$.
We then write
\begin{equation}
	\label{eq1249tri}
\int_{\CC \setminus B_{r}(a)} \frac{\dif\mu(z)}{z - a}
= \int_{\CC \setminus Q_n} \frac{\dif\mu(z)}{z - a} + \int_{Q_n \setminus B_{r}(a)} \frac{\dif\mu(z)}{z - a}.
\end{equation}
Note that
\begin{equation}
	\label{eq1249q}
\int_{Q_n \setminus B_{r}(a)} \frac{\dif\mu(z)}{\abs{z - a}}
\le \frac{\mu(Q_n)}{r}
\le \frac{1}{2\delta} \frac{1}{4^{n}\sigma_n}.
\end{equation}
Indeed, as there exists $x \in (B_{r}(a) \cap E(\Lambda)) \setminus Q_{n+1}$\,, using \eqref{eq1253} we have
\begin{equation*}
r \ge |x-a| 	
\ge \sigma_n -2\sigma_{n+1}
= \sigma_{n}(1 - 2\lambda_{n+1}) 
\ge 2 \delta \sigma_n.
\end{equation*}
Since the right-hand side of \eqref{eq1249q} converges to zero as \(n \to \infty\), we have from \eqref{eq1249tri} that the existence of the limit \eqref{eq1249bis} implies that of \eqref{eq1249}.

We investigate the existence of \eqref{eq1249bis} using the Cauchy criterion.
By additivity of the integral, for every \(m > n \ge 1\) we may write
\begin{equation}
\label{eqAdditivity}
\int_{Q_{n} \setminus Q_{m}}{\frac{\dif\mu(z)}{z - a}}
= \sum_{j = n}^{m - 1}{\int_{Q_{j} \setminus Q_{j+1}}{\frac{\dif\mu(z)}{z - a}}}.
\end{equation}
We now focus on the integrals in the right-hand side:
	
\begin{lemma}
	\label{lemmaRescalling}
	Let \(a\) be the alternating point of \(E(\Lambda)\).{}
	Then, for every \(j \in \N_{*}\),{}
	\[{}
	\int_{Q_{j} \setminus Q_{j + 1}} \frac{\dif\mu(z)}{z - a}
	= \frac{(-1)^{j}}{4^{j} \sigma_{j}}	\int_{Q_{0} \setminus \widetilde Q_{1, j}} \frac{\dif\widetilde\mu_{j}(z)}{z - \widetilde a_{j}},
	\]
	where \(\widetilde a_{j}\) is the alternating point of the Cantor set \(E(\widetilde\Lambda_{j})\) associated with the shifted sequence \(\widetilde\Lambda_{j} = (\lambda_{j+1}, \lambda_{j+2}, \ldots)\), \(\widetilde Q_{1, j}\) is the square of generation \(1\) with side length \(\lambda_{j+1}\) that contains \(\widetilde a_{j}\) in the construction of \(E(\widetilde\Lambda_{j})\) and \(\widetilde\mu_{j}\) is the intrinsic probability measure associated with \(E(\widetilde\Lambda_{j})\).
\end{lemma}	

\begin{proof}
	We translate the squares \(Q_{j}\) and \(Q_{j+1}\) using the map \(z \mapsto z - b_{j}\)\,, where \(b_{j}\) is the lower left-hand vertex of \(Q_{j}\) and then rescale the resulting sets with  \(z \mapsto z/\sigma_{j}\).{}
	As a result \(Q_{j}\)\,, with side length \(\sigma_{j}\)\,, is transformed to the unit square \(Q_{0}\).{}
	The combination of both maps transforms the measure \(\mu\) in \(Q_{j}\) to the rescaled measure \(\widetilde\mu_{j}/4^{j}\) in \(Q_{0}\). 
	When \(j\) is even, the point \(a\) is mapped to \(\widetilde a_{j}\) since \(Q_{j + 1}\) and \(\widetilde Q_{1, j}\) are both squares with relative position I inside \(Q_{j}\) and \(Q_{0}\), respectively.
	We then get
	\[{}
	\int_{Q_{j} \setminus Q_{j + 1}} \frac{\dif\mu(z)}{z - a}
	= \frac{1}{4^{j} \sigma_{j}}	\int_{Q_{0} \setminus \widetilde Q_{1, j}} \frac{\dif\widetilde\mu_{j}(z)}{z - \widetilde a_{j}}.
	\]
	In contrast, when \(j\) is odd, \(Q_{j + 1}\) is a square with position III inside \(Q_{j}\)\,, while \(\widetilde Q_{1, j}\) has position I in \(Q_{0}\). 
	Thus, \(a\) is mapped to another point \(\widehat a_{j}\) in \(E(\widetilde\Lambda_{j})\) that arises using the reverted positions III, I, III, I, \(\ldots\)\,
	We have in this case
	\[{}
	\int_{Q_{j} \setminus Q_{j + 1}} \frac{\dif\mu(z)}{z - a}
	= \frac{1}{4^{j} \sigma_{j}} \int_{Q_{0} \setminus \widehat Q_{1, j}} \frac{\dif\widetilde\mu_{j}(z)}{z - \widehat a_{j}},
	\]
	where \(\widehat Q_{1, j}\) is the square of generation \(1\) in \(Q_{0}\) with side-length \(\lambda_{j+1}\) at position III.
	Using the reflection \(R : z \mapsto (1+i) - z\) around the center \((1+i)/2\) of \(Q_{0}\)\,, we have
	\(R(\widehat Q_{1, j}) = \widetilde Q_{1, j}\) and \(R(\widehat a_{j}) = \widetilde a_{j}\).{}
	Since \(\widetilde\mu_{j}\) is invariant under \(R\) by symmetry of \(E(\widetilde\Lambda_{j})\), by a change of variable in the integral we then get
	\[{}
	\int_{Q_{0} \setminus \widehat Q_{1, j}} \frac{\dif\widetilde\mu_{j}(z)}{z - \widehat a_{j}}
	= \int_{Q_{0} \setminus \widetilde Q_{1, j}} \frac{\dif\widetilde\mu_{j}(w)}{(1 + i) - w - \widehat a_{j}}
	= \int_{Q_{0} \setminus \widetilde Q_{1, j}} \frac{\dif\widetilde\mu_{j}(w)}{\widetilde a_{j} - w},
	\]
	which gives the conclusion when \(j\) is odd.
\end{proof}

We assume henceforth that \(a\) is the alternating point of \(E(\Lambda)\).
By \eqref{eqAdditivity} and Lemma~\ref{lemmaRescalling}, we have
\begin{equation}
	\label{eq1373}
\int_{Q_{n} \setminus Q_{m}}{\frac{\dif\mu(z)}{z - a}}
= \sum_{j = n}^{m - 1}{A_{j}I_{j}},
\end{equation}
where
\[{}
A_{j} \vcentcolon= \frac{(-1)^{j}}{4^{j} \sigma_{j}}
\quad \text{and} \quad{}
I_{j} \vcentcolon= \int_{Q_{0} \setminus \widetilde Q_{1, j}} \frac{\dif\widetilde\mu_{j}(z)}{z - \widetilde a_{j}}.
\]
Observe that the sequence \((I_{j})_{j \in \N_{*}}\) is bounded.
Indeed, the squares of generation \(1\) in the construction of \(E(\widetilde\Lambda_{j})\) are away from each other by a distance of at least \(1 - 2\lambda_{j + 1} \ge 2\delta\).{}
Thus,
\begin{equation}
\label{eq1386}
\abs{I_{j}}
\le \frac{\widetilde\mu_{j}(Q_{0})}{1 - 2\lambda_{j + 1}}
\le \frac{1}{2\delta}.
\end{equation}

Using Abel's summation method, we write the sum in \eqref{eq1373} as follows
\begin{equation}
	\label{eq1413}
\int_{Q_{n} \setminus Q_{m}}{\frac{\dif\mu(z)}{z - a}}
= \sum_{j = n}^{m - 2}{\biggl(\, \sum_{k = n}^{j}{A_{k}} \biggr)(I_{j} - I_{j+1})} + \biggl(\,\sum_{k = n}^{m-1}{A_{k}} \biggr)I_{m-1}.
\end{equation}
By assumption, the sequence \(\bigl( (-1)^{k}A_{k} \bigr)_{k \in \N}\) is non-increasing.
From a property of alternate sums,
\[{}
0 \le (-1)^{n}\sum_{k = n}^{j}{A_{k}} 
\le (-1)^{n}A_{n}
= \frac{1}{4^{n} \sigma_{n}}.
\]
It then follows from \eqref{eq1386} and \eqref{eq1413} that
\begin{equation}
	\label{eqEstimateCauchy}
\biggl|\int_{Q_{n} \setminus Q_{m}}{\frac{\dif\mu(z)}{z - a}}\biggr|
\le \frac{1}{4^{n} \sigma_{n}} \biggl(\, \sum_{j = n}^{m - 2}{\abs{I_{j} - I_{j+1}}} + \frac{1}{2\delta} \biggr).
\end{equation}

\begin{lemma}
	\label{lemmaEstimateI}
	For every \(j \in \N\), we have
	\[{}
	\abs{I_{j} - I_{j+1}}
	\le C \sum_{l=j+1}^{\infty}{\frac{1}{2^{l-j}} \abs{\lambda_{l} - \lambda_{l+1}}}.
	\]
\end{lemma}

\resetconstant
\begin{proof}
	Without loss of generality we may assume that \(j = 0\), in which case \(\widetilde\Lambda_{0} = \Lambda\) and \(\widetilde{Q}_{1, 0} = Q_{1}\).
	Given \(k \in \N_{*}\), we recall that \(S_{k, \alpha}\) with \(\alpha \in \{1, \ldots, 4^{k}\}\) are the squares of side length \(\sigma_{k}\) of generation \(k\) that are used in the construction of \(E(\Lambda)\). 
	We may label them so that, for \(\alpha \in \{1, \ldots, 3 \cdot 4^{k -1}\}\), the squares \(S_{k, \alpha}\) cover \(E(\Lambda) \setminus Q_{1}\). 
	Denote the center of \(S_{k, \alpha}\) by \(c_{k, \alpha}\).{}
	Since the measure \(\widetilde\mu_{0} = \mu\) is supported by \(E(\Lambda)\) and \(\mu(S_{k, \alpha}) = 1/4^{k}\) for every \(\alpha\), by additivity of the integral we have
	\[{}
	I_{0} 
	= \int_{Q_{0} \setminus Q_{1}}{\frac{\dif\mu(z)}{z - a}}
	= \sum_{\alpha = 1}^{3 \cdot 4^{k-1}} \int_{S_{k, \alpha}} \Bigl( \frac{1}{z - a} - \frac{1}{c_{k, \alpha} - a} \Bigr) \dif\mu(z)
	+ \frac{1}{4^{k}} \sum_{\alpha = 1}^{3 \cdot 4^{k-1}}{\frac{1}{c_{k, \alpha} - a}}.
	\]
	Since the distance between \(S_{k, \alpha}\) and \(Q_{1}\) is larger than \(1 - 2\lambda_{1} \ge 2\delta\), applying the Mean Value Theorem for every \(z \in S_{k, \alpha}\)\,,
	\[{}
	\Bigabs{\frac{1}{z - a} - \frac{1}{c_{k, \alpha} - a}}
	\le \C \abs{z - c_{k, \alpha}}
	\le \Cl{cte-1171} \, \sigma_{k}\,,
	\]
	where the constants involved depend on \(\delta\).
	Since \(\lambda_{k} \le 1/2\) for every \(k\), we have that \(\sigma_{k} \le 1/2^{k}\) and then
	\begin{equation}
		\label{eq1470}
	\biggl| I_{0} - \frac{1}{4^{k}} \sum_{\alpha = 1}^{3 \cdot 4^{k-1}}{\frac{1}{c_{k, \alpha} - a}} \biggr|{}
	 \le \sum_{\alpha = 1}^{3 \cdot 4^{k-1}} \int_{S_{k, \alpha}} \Bigabs{ \frac{1}{z - a} - \frac{1}{c_{k, \alpha} - a}} \dif\mu(z)
	 \le 3 \cdot 4^{k-1} \cdot \frac{\Cr{cte-1171} \, \sigma_{k}}{4^{k}}
	\le \frac{\Cr{cte-1171}}{2^{k}}.
	\end{equation}
	Analogously, let us denote by \(S_{k, \alpha}'\) the squares of side length \(\sigma_{k}' \vcentcolon= \lambda_{2} \cdots \lambda_{k+1}\) of generation \(k\) that are used in the construction of \(E(\widetilde\Lambda_{1})\) and by \(c_{k, \alpha}'\) their centers.
	We label these squares so that \(S_{k, \alpha}\) and \(S_{k, \alpha}'\) have the same relative position inside \(Q_{0}\) for every \(\alpha \in \{1, \ldots, 4^{k}\}\).{}
	In particular, if \((\alpha_{k})_{k \in \N_{*}}\) is the sequence of integers such that \(a \in S_{k, \alpha_{k}}\) for every \(k \in \N_{*}\)\,, then 
	\begin{equation}
		\label{eqConvergenceCenters}
	c_{k, \alpha_{k}} \to a 
	\quad \text{and} \quad 
	c_{k, \alpha_{k}}' \to \widetilde a_{1}.
	\end{equation}
	Using the counterpart of \eqref{eq1470} for the shifted sequence \(\widetilde\Lambda_{1}\),
	by the triangle inequality we get
	\[{}
	\abs{I_{0} - I_{1}}
	\le \frac{1}{4^{k}} \sum_{\alpha = 1}^{3 \cdot 4^{k-1}}{\Bigabs{\frac{1}{c_{k, \alpha} - a} - \frac{1}{c_{k, \alpha}' - \widetilde{a}_{1}}}} + \frac{\Cr{cte-1171}}{2^{k-1}}.
	\]
	By the Mean Value Theorem, for every \(\alpha \in \{1, \ldots, 3 \cdot 4^{k -1}\}\),
	\[{}
	\Bigabs{\frac{1}{c_{k, \alpha} - a} - \frac{1}{c_{k, \alpha}' - \widetilde{a}_{1}}}
	\le \Cl{cte-1197} \bigl( \abs{c_{k, \alpha} - c_{k, \alpha}'} + \abs{a - \widetilde{a}_{1}} \bigr).
	\]
	Thus,
	\begin{equation}
		\label{eq1500}
	\abs{I_{0} - I_{1}}
	\le \frac{\Cr{cte-1197}}{4^{k}} \sum_{\alpha = 1}^{3 \cdot 4^{k-1}}{\abs{c_{k, \alpha} - c_{k, \alpha}'}} + \Cr{cte-1197}\abs{a - \widetilde{a}_{1}} + \frac{\Cr{cte-1171}}{2^{k-1}}.
	\end{equation}
	
	\begin{Claim}
		For every \(k \in \N_{*}\) and every \(\alpha \in \{1, \ldots, 4^{k}\}\),{}
		\begin{equation}
			\label{eq1508}
		\abs{c_{k, \alpha} - c_{k, \alpha}'}
		\le 4 \sqrt{2} \sum_{l = 1}^{k}{\frac{1}{2^{l}} \abs{\lambda_{l} - \lambda_{l+1}}}.
		\end{equation}
	\end{Claim}
	
	\begin{proof}[Proof of the Claim]
		We first show that each coordinate of \(c_{k, \alpha}\) is a polynomial of the form
		\begin{equation}
			\label{eqPolynomial}
		P_{k, \alpha}(\lambda_1, \dots, \lambda_k) = \beta_0 + \beta_1 \,\lambda_1 + \beta_2 \, \lambda_1 \lambda_2+ \cdots + \beta_k \,
		\lambda_1 \cdots \lambda_k,
		\end{equation}
		where \(\beta_j \in \{0, -1, 1\}\) for every \(j \in \{0, 1, \dots, k-1\}\) and  $\beta_k \in \{-1/2, 1/2\}$.{}
		Let us focus for example on the first component.
		The first coordinate of \(c_{1, \alpha}\) is of the form 
		\[{}
		\frac{1}{2} \pm \frac{1}{2} (1 - \lambda_{1})
		= \frac{1}{2} \pm \frac{1}{2}  \mp \frac{1}{2} \lambda_{1}.
		\]
		so the conclusion holds for \(k = 1\).{}
		We now take any \(k \ge 1\), a point \(c_{k+1, \alpha}\) and the square \(S_{k, \widetilde\alpha}\) of generation \(k\) that contains it, including three other centers of generation \(k + 1\).{}
		Then, assuming that the first component  \(x_{k}\) of \(c_{k, \widetilde\alpha}\) is of the form \eqref{eqPolynomial}, the first coordinate of \(c_{k + 1, \alpha}\) is then given by 
		\[
		x_{k} \pm \frac{1}{2} (\sigma_{k} - \sigma_{k+1})
		= \beta_0 + \beta_1 \,\lambda_1 + \beta_2 \, \lambda_1 \lambda_2+ \dots + \Bigl(\beta_k \pm \frac{1}{2}\Bigr) \,
		\lambda_1 \cdots \lambda_k \mp \frac{1}{2} \lambda_{1} \cdots \lambda_{k+1},
		\]
		which also has the form \eqref{eqPolynomial} for \(k+1\) as we wanted to show.{}
		
		The coefficients of the polynomials \(P_{k, \alpha}\) associated with the components of \(c_{k, \alpha}\) depend solely on \(k\) and \(\alpha\). Hence, the components of \(c_{k, \alpha}'\) are given by the same polynomials, but evaluated at \((\lambda_{2}, \dots, \lambda_{k+1})\).{}
		For example, if \(x_{k}\) and \(x_{k}'\) denote one of their components, then
		\[{}
		\begin{split}
		x_{k} - x_{k}'
		& = P_{k, \alpha}(\lambda_{1}, \dots, \lambda_{k}) - P_{k, \alpha}(\lambda_{2}, \dots, \lambda_{k+1})\\
		& = \beta_{1}(\lambda_{1} - \lambda_{2}) + \beta_{2} (\lambda_{1}\lambda_{2} - \lambda_{2} \lambda_{3})
		+ \cdots + \beta_{k}(\lambda_{1} \cdots \lambda_{k} - \lambda_{2} \cdots \lambda_{k+1})\\
		& = \beta_{1}(\lambda_{1} - \lambda_{2}) + \beta_{2} \lambda_{2} (\lambda_{1} - \lambda_{3})
		+ \cdots + \beta_{k} \lambda_{2} \cdots \lambda_{k} (\lambda_{1} - \lambda_{k+1}).
		\end{split}
		\]
		Using the triangle inequality and the fact that \(\lambda_{j} \le 1/2\) for each \(j\), we get
\[
\begin{split}
| c_{k, \alpha} - c_{k, \alpha}' | 
& \le \sqrt{2}\bigl( \abs{\lambda_1 - \lambda_2} + \lambda_2 \abs{\lambda_1 - \lambda_3} + \cdots + \lambda_2\lambda_3\cdots\lambda_{k} \abs{\lambda_1 - \lambda_{k+1}} \bigr)\\
& \le 2\sqrt{2} \sum_{j = 1}^{k}{\frac{1}{2^{j}}\abs{\lambda_1 - \lambda_{j + 1}}}.
\end{split}
\]
Using that \(\abs{\lambda_1 - \lambda_{j + 1}} \le \sum\limits_{l = 1}^{j}{\abs{\lambda_{l} - \lambda_{l+1}}}\) and interchanging the order of summation, the claim follows.
	\end{proof}

Applying estimate \eqref{eq1508} with the sequence of integers \((\alpha_{k})_{k \in \N_{*}}\) that satisfies \eqref{eqConvergenceCenters} and letting \(k \to \infty\), we deduce that
	\[{}
	\abs{a - \widetilde a_{1}}
	\le 4\sqrt{2} \sum_{l = 1}^{\infty}{\frac{1}{2^{l}} \abs{\lambda_{l} - \lambda_{l+1}}}.
	\]
By the Claim and \eqref{eq1500} we thus have, for every \(k \ge 1\),{}
	\[{}
	\abs{I_{0} - I_{1}}
	\le \C \sum_{l = 1}^{\infty}{\frac{1}{2^{l}} \abs{\lambda_{l} - \lambda_{l+1}}} + \frac{\Cr{cte-1171}}{2^{k-1}}.
	\]
	Letting \(k \to \infty\), the lemma follows.
\end{proof}

Combining \eqref{eqEstimateCauchy} and Lemma~\ref{lemmaEstimateI}, for every \(m > n \ge 1\) we get
\[
\biggl|\int_{Q_{n} \setminus Q_{m}}{\frac{\dif\mu(z)}{z - a}}\biggr|
\le \frac{\Cl{cte-1282}}{4^{n} \sigma_{n}} \biggl(\,\sum_{j = n}^{\infty}{\sum_{l=j+1}^{\infty}{\frac{1}{2^{l-j}} \abs{\lambda_{l} - \lambda_{l+1}}}} + \frac{1}{2\delta} \biggr){}.
\]
Interchanging the order of summation, we deduce that
\[{}
\biggl|\int_{Q_{n} \setminus Q_{m}}{\frac{\dif\mu(z)}{z - a}}\biggr|
\le \frac{\Cr{cte-1282}}{4^{n} \sigma_{n}} \biggl(\,\sum_{l = n+1}^{\infty}{\abs{\lambda_{l} - \lambda_{l+1}}} + \frac{1}{2\delta} \biggr){}.
\]
Recalling that \(\sum\limits_{l = 1}^{\infty}{\abs{\lambda_{l} - \lambda_{l+1}}} < \infty\), we then have
\[{}
\biggl|\int_{Q_{n} \setminus Q_{m}}{\frac{\dif\mu(z)}{z - a}}\biggr|{}
\le \frac{\C}{4^{n} \sigma_{n}}
\quad \text{for every \(m > n \ge N\).}
\]
Since \(4^{n} \sigma_{n} \to \infty\) as \(n \to \infty\), we conclude that the limit in \eqref{eq1249bis} exists and then the Cauchy integral of \(\mu\) has a principal value at \(a\).{}

To complete the proof, it remains to consider the case where \eqref{eq1253} fails, which implies that \(\lambda_{n} \to 1/2\) as \(n \to \infty\).{}
The argument here does not require the monotonicity of \(\Lambda\).
Indeed, from the convergence of \(\Lambda\) to \(1/2\), for every \(1 < s < 2\) there exists a constant \(\Cl{cte1601} > 0\) such that
\[{}
\mu(B_{r}(a)) 
\le \Cr{cte1601} \, r^{s}
\quad \text{for every \(a \in \CC\) and \(r > 0\)}.
\]
Given any \(a \in \CC\), by Cavalieri's principle one then gets
\[{}
\int_{\CC}\frac{\dif\mu(z)}{|z - a|}
= \int_{0}^{\infty} \frac{\mu(B_{r}(a))}{r^{2}} \dif r
 < \infty.
\]
Hence, the principal value of \(C(\mu)\) exists everywhere.
\qed


\section*{Acknowledgements}

The authors would like to thank L.~Grafakos and P. Mattila for an interesting correspondence.
J.~Cuf\'i acknowledges partial support from grant 2017SGR358 (Generalitat de Catalunya) and J.~Verdera from grants 2017SGR395
(Generalitat de Catalunya),  MTM2016--75390 and PID2020-112881GB-I00 (Mi\-nis\-terio de
Educaci\'{o}n y Ciencia).
A.~C. Ponce is grateful for the invitation, hospitality and support of the Departament de Matemàtiques of the Universitat Autònoma de Barcelona.
He also acknowledges support of the Fonds de la Recherche scientifique (F.R.S.--FNRS) from grant J.0020.18.


\end{document}